%% file: 2025_10_29_final_submission.tex
\documentclass[conf]{new-aiaa} 
\usepackage[utf8]{inputenc}

\usepackage{graphicx}

\usepackage{amsmath}
\input{PackagesCommands}
\usepackage[colorinlistoftodos,textwidth=1 in,textsize=footnotesize]{todonotes}
\title{Adaptive Thrust Regulation in Solid-fuel Ramjet \\ with Variable Geometry Inlet}

\author{\large Parham Oveissi, Ryan DeBoskey, Venkateswaran Narayanaswamy, and Ankit Goel
\footnote{Ryan DeBoskey and Venkateswaran Narayanaswamy are with the Department of Mechanical Engineering, North Carolina State University, Raleigh, NC 27695 {\tt\small rddebosk@ncsu.edu, vnaraya3@ncsu.edu }}
\footnote{Parham Oveissi and Ankit Goel are with the Department of Mechanical Engineering, University of Maryland, Baltimore County,1000 Hilltop Circle, Baltimore, MD 21250. {\tt\small parhamo1@umbc.edu, ankgoel@umbc.edu }}
}

\begin{document}
\maketitle

\begin{abstract}
    This paper presents the application of a novel data-driven adaptive control technique, dynamic mode adaptive control (DMAC), to regulate thrust in a solid-fuel ramjet (SFRJ).
    A quasi-static one-dimensional model of SFRJ with a variable geometry inlet is developed to compute thrust.
    An adaptive tracking controller is then designed using the DMAC framework, which leverages dynamic mode decomposition to approximate the local system behavior, followed by a tracking controller designed around the identified model.
    Simulation results demonstrate that DMAC achieves accurate thrust regulation across a range of commanded profiles and operating conditions, without requiring an analytical model of the SFRJ. 
    These findings indicate that DMAC provides a reliable and effective approach for model-free thrust regulation in an SFRJ with variable-geometry inlets as the control input.
\end{abstract}

\section{INTRODUCTION}
\label{sec:introduction}

Solid-fuel ramjets (SFRJ) promise greater propulsive performance and capability compared to traditional engines due to the lack of complex turbomachinery and the use of efficient high-density solid field \cite{Schulte1986}.   
%
However, precise control of SFRJs remains an open challenge due to complex combustion dynamics, which involve interactions between turbulent flame dynamics, solid-fuel pyrolysis, and complex hydrocarbon chemistry.  
A review of the progress and challenges in understanding the SFRJ combustion phenomena is compiled in several review articles \cite{KRISHNAN1998219,Gany_2009,Veraar_2022}.  

The traditional approach of regulating the SFRJ thrust is based on the bypass air channels \cite{natan1993experimental,pelosi2003bypass}. 
Experimentally, bypass ratio control has shown increased combustor efficiency \cite{evans2023performance}, but requires the addition of bypass piping, which increases the complexity of the SFRJ geometry whilst decreasing the theoretical capacity for solid-fuel loading. 
Recent work has investigated the potential of using variable inlet design as an alternative means for regulating SFRJ performance \cite{deboskey2025augmentation}.  
Variable inlet control presents potential advantages by minimizing the actuator hardware footprint and providing a novel mechanism to regulate air ingested for combustion and, thus, thrust.
However, due to the complex dynamic relationship between the inlet geometry and the thrust generated, which is sensitive to geometric variations of the SFRJ, as well as to variations in flow conditions and flight conditions, designing a control system to regulate the SFRJ thrust remains a challenging problem. 
The focus of this work is thus the investigation of an adaptive control technique to regulate the thrust generated by an SFRJ.

In this work, a quasi-static one-dimensional model of SFRJ developed in \cite{deboskey2025augmentation} is used. 
The control design approach used in this work is based on the dynamic mode adaptive control (DMAC) algorithm, which is described in detail in \cite{oveissi2025modelfreedynamicmodeadaptive}.
Although our previous work based on the retrospective cost adaptive control (RCAC) framework has been successfully applied to the thrust control problem in ramjets \cite{oveissi2025adaptive,oveissi2024adaptive,oveissi2023learning, goel2019output,goel2018retrospective, goel2015scramjet,deboskey2025situ, oveissi2025learning,manuel2025real}, RCAC requires the choice of a \textit{target filter} that captures the essential modeling information required to update the control law. 
However, the fixed filter choice limits the system's operational envelope and significantly complicates the design process in a multi-input, multi-output system.
In contrast, DMAC does not require any modeling information and instead uses measured data to identify a low-order dynamic approximation of the system and an adaptive linear controller, potentially enabling a larger operational envelope. 


The paper is organized as follows. 
Section \ref{sec:numerical_model} describes the numerical SFRJ model used to predict SFRJ thrust. 
Section \ref{sec:DMAC} briefly reviews the dynamic mode adaptive control framework to regulate the thrust of the SFRJ model.
Section \ref{sec:prelim_results} presents simulation results to demonstrate the application of the DMAC technique to regulate the SFRJ thrust. 
Finally, the paper concludes in Section \ref{sec:conclusions}.

\section{SFRJ Model}

\label{sec:numerical_model}
This study considers steady cruise operation of a model 140 mm external diameter SFRJ at a flight Mach number of 3.25 and an altitude of 30 km.
A constant freestream temperature and pressure are assumed based on the 1976 US Standard Atmosphere model \cite{standard1976atmosphere}.  Table \ref{cruise_condition} provides details of the selected cruise condition. 

\begin{table}[H]
\centering
\caption{Overview of selected cruise condition}\label{cruise_condition}
\begin{tabular}{ c  c  c  c }
\hline \hline
$M$ & $H$ [km] & $P_{t0}$ [Pa] & $T_{t0}$  [K]\\
\hline
3.25 & 30 & 63677 & 748.6  \\
\hline
\end{tabular}
\end{table}


This work considers a static quasi-one-dimensional feed-forward SFRJ with variable cowl control, modified based on the model developed by DeBoskey et al. \cite{deboskey2025augmentation}.  Figure \ref{IntegratedSystemDiagram} shows a schematic of an SFRJ platform with thermodynamic stations labeled.  Freestream air, station ``0'', is ingested and compressed through the inlet and isolator system, station ``1'', and enters the combustor entrance, station ``2''.  The solid fuel is embedded within the combustion chamber in station ``3''; the port radius $r_3$ increases as the fuel is consumed during combustion.  
Additional mixing and combustion occurs at the combustor aft-end, station ``4'', before expansion and exhaust through the nozzle, station ``e''. 

\begin{figure}[h]
\centering
    \includegraphics[width=1\columnwidth]{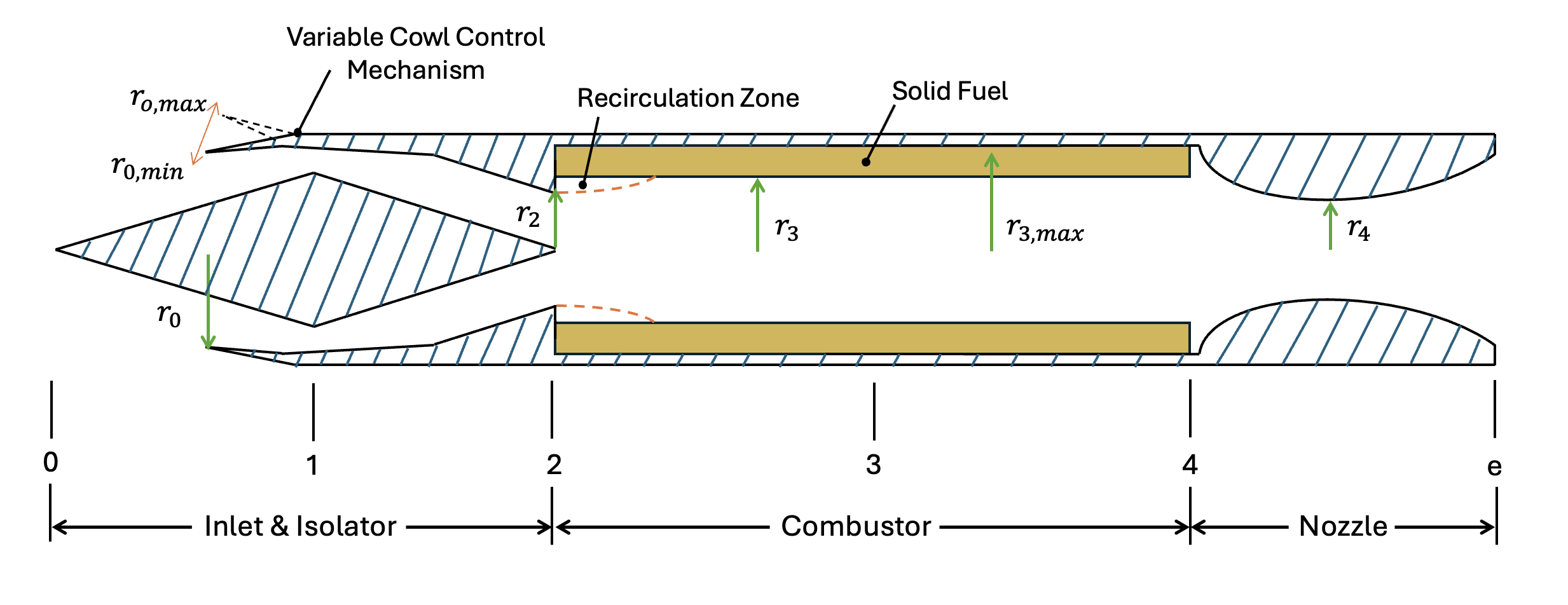}
    \caption{Schematic of model SFRJ projectile with thermodynamic stations labeled.  System performance is controlled through the translation of a variable cowl mechanism \cite{deboskey2025augmentation}.}
    \label{IntegratedSystemDiagram}
\end{figure}

A generic axisymmetric nose-spike center body inlet with a variable geometry cowl is selected for the design. 
The nose-spike and cowl design loosely resembles the NASA 1507 inlet \cite{slater2024wind}. 
The variable geometry mechanism adjusts both the location of the cowl tip in the $x$-direction and the angle of the cowl.
This corresponds to a translation in the cowl tip along the conical shock angle of the center body at the cruise Mach condition. 
The translation is represented in terms of the cowl radius, $r_0$, which is equivalent to the freestream capture area radius.  
This approach is similar to other cowl-based variable geometry methods in which the cowl angle and cowl location are investigated as separate operations \cite{Teng.et.al-CowlLocation,dalle2011performance,Reardon.et.al-CompUnstart,Liu.et.al-Restart,RamprakashMuruganandam-ExpUnstart,TengJianHuacheng-VariableCowl}.  
In \cite{deboskey2025augmentation}, a series of parametric 2D axisymmetric Reynolds-Averaged Navier-Stokes (RANS) simulations were performed on a structured two-dimensional grid with an increasing freestream capture area.  Isolator exit mass flow rate, mass-averaged stagnation pressure, and mass-averaged Mach number were calculated and curve-fitted as a function of $r_0.$
This model is extended to work across a range of altitudes by normalizing the isolator exit quantities based on freestream conditions.

The isolator exit conditions are then fed forward to a quasi-one-dimensional combustor model.
The regression rate of the solid fuel, $\dot{r}$, is assumed to depend on the inlet mass flux, $G$, total temperature, $T_{t2}$, and combustor pressure, $P_4$, and is given by
\begin{equation}\label{regression}
    \dot{r} =  \alpha G_{a}^{a} P_4^{b} T_{t2}^{c} ,
\end{equation}
\noindent where $\alpha$, $a$, $b$, and $c$ are empirical curve fits from Vaught et al. \cite{vaught1992investigation}.  This work considers a hypothetical model hydroxyl-terminated polybutadiene (HTPB) polymer that is completely converted to the gaseous hydrocarbon, 1,3-butadiene ($C_4 H_6$).  The model assumes that the fuel grain only burns radially, with uniform regression along the axial length of the combustor.
To calculate the combustor aft-end mixing pressure for \eqref{regression}, a simple friction correlation is used to model internal friction, which implies that 
\begin{equation}\label{pressuredrop}
    \Delta P_{t,2-4} = \frac{f_D}{4} \frac{L_f}{d} \frac{\rho_3 u_3^2 }{2}  ,
\end{equation}
where $\Delta P_{t,2-4}$ is the pressure change between station ``2'' and station ``4'', $f_D$ is the Darcy friction factor, $L_f$ is the length of the fuel grain, $\rho$ is the port fluid density, and $u_3$ is the port velocity.

Due to the solid-fuel regression, the solid-fuel port radius, $r_3$, is constantly increasing during SFRJ operation and is updated based on \cite{Hadar1992,evans2023performance} as
\begin{equation}
    r_{3}(t_{i+1}) = r_{3}(t_{i}) - \dot{r}(t_i) (t_{i+1}-t_{i}).
\end{equation}
The simulations are time-integrated until the port radius reaches the maximum allowable radius, $r_{3,\rm max}$. Table \ref{nominalgeo} details the SFRJ geometry considered in this work. 
\begin{table}[h] 
 	\centering
 	\caption{Description of SFRJ geometry}
 	\begin{tabular}{ c c c c c}
 		\hline \hline
 		$r_{0}$ [mm] & $r_{2}$ [mm] & $r_3$ &  $r_{t}$ [mm] & $L_{f}$ [mm]\\
 		\hline
 		[47.88,59.28] & 46.7 & [59.2,68.6] & 50.4 & 500\\
            \hline
 	\end{tabular}
 	\label{nominalgeo}
\end{table}

At each time-step, the fuel mass flow rate of 1,3-butadiene, $\dot{m}_f$, is given by
\begin{equation}
    \dot{m}_f 
        =
            A_f \rho_f \dot{r},
\end{equation}
where $A_f = 2 \pi r_3 L_f$ is the exposed surface area of the fuel grain and $\rho_f$ = 900 kg/m$^3$ is the density of HTPB.
Note that a global equivalence ratio, $\phi_G$ can be calculated by
\begin{equation}
    \phi_G = \frac{f}{f_{stoich}} = \frac{\dot{m}_f / \dot{m}_{air}}{f_{stoich}} ,
\end{equation}
where $f_{stoich}$ is the stoichiometric fuel-to-air mass flow ratio.  
The equilibrium flame temperature, $T_{4,\rm eq}$, and the ratio of specific heats, $\gamma_{4}$, at the aft mixing end are calculated using an equilibrium Gibbs solver, assuming constant enthalpy and pressure, in CANTERA \cite{cantera}.  The initial mixture composition is determined by the calculated global equivalence ratio $\phi_G$, the inlet air temperature $T_2$, and the aft-end pressure $P_4$ .  
A pressure-comprehensive skeletal kinetics mechanism for 1,3-butadiene combustion developed by Ciottoli et al. \cite{Ciottoli2017} is used to calculate equilibrium products.
Previous work has validated the skeletal mechanism against detailed 1,3-butadiene combustion mechanisms \cite{deboskey_mechs_2024,DEBOSKEY2025CNF}.  The final static temperature at the aft-mixing end of the SFRJ combustor, $T_{4}$, is calculated, assuming a constant combustion efficiency of $\eta_c$ = 0.75, as
\begin{equation}
      T_{4} = \eta_c ( T_{4,\rm eq} - T_{2}) + T_{2}.
\end{equation}


Using isentropic expansion to ambient pressure, the aft-end combustor total pressure and temperature are used to determine a theoretical exhaust velocity, $u_{\rm e,th}$, which is computed as
\begin{align}
    u_{\rm e,th} 
        =
            \big[ 
                2\gamma_4 R_4 T_{t4}(1/(\gamma_4-1)) 
                (1 - (P_0 / P_{t4})^{(\gamma_4-1)/\gamma_4})
            \big]^{1/2},
\end{align}
\noindent where $R_4$ and $\gamma_4$ are the specific gas constant and ratio of specific heats, respectively, calculated using the equilibrium composition from station ``4'' in CANTERA.  
This work assumes an idealized nozzle that expands gas at the combustor exit to ambient pressure and operates with a nozzle efficiency, $\eta_n$ = 0.95, to calculate the actual exhaust velocity, $u_{e}$.  
Finally, the generated thrust, $T$, is given by
\begin{equation}\label{thrustcalc}
    T = \dot{m}_{air} (1 + f) u_e - \dot{m}_{air} u_0.
\end{equation}

\section{Dynamic Mode Adaptive Control}
\label{sec:DMAC}
This section briefly reviews the dynamic mode adaptive control (DMAC) algorithm, presented in \cite{oveissi2025modelfreedynamicmodeadaptive}. 
As shown in Figure \ref{fig:DMAC_architecture}, consider a dynamic system in a basic servo loop architecture whose input is $u_k \in \BBR^{l_u}$ and the output is $y_k \in \BBR^{l_y}.$
The objective of the DMAC controller is to generate a discrete-time input signal $u_k$ such that the sampled output $y_k$ tracks the reference signal $r_k.$
The DMAC controller consists of a dynamic mode approximation to approximate the local system behavior and a tracking controller designed based on the identified input-output model.

    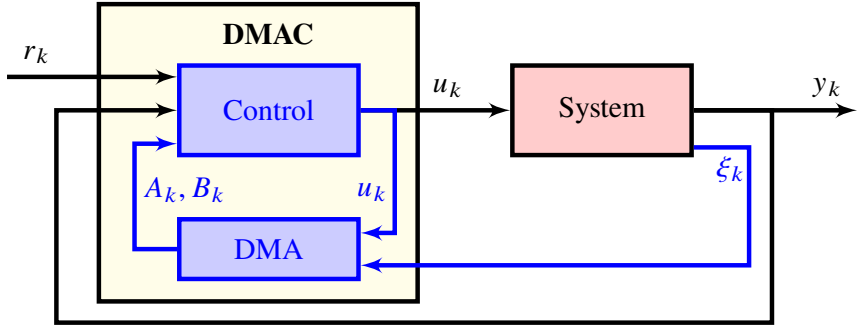
\begin{figure}[htbp]
    \centering
    \resizebox{0.7\columnwidth}{!}
    {
    \begin{tikzpicture}[auto, node distance=2cm,>=latex',text centered, line width = 1.5]

        \draw[draw=black, fill=yellow!10] (-2,-2.25)              
             rectangle ++(3.75,3.5) node [xshift=-5.0em, yshift=-1em] {\textbf{DMAC}} ;

             
        \node [smallblock, blue, fill = blue!20, minimum width=6em, minimum height=3em] (Controller) {Control};

        \node [smallblock, , fill = red!20, right = 5 em of Controller, minimum width = 6em, minimum height=3em] (plant) {System};

        

        
        \node [smallblock, blue, fill = blue!20, below = 2 em of Controller, minimum width=6em] (DMA) {DMA};

        \draw[<-] (Controller.160) -- +(-2,0) node[xshift = 1em, yshift = 0.75em]{$r_k$};
        \draw[->] (plant) -- +(3,0) node[xshift = -1em, yshift = 0.75em]{$y_k$};        
        \draw[->] (Controller) node[xshift = 6em, yshift = 0.75em]{$u_k$} -- (plant);

        \draw[->] (plant) -| +(2,-2.5) |- (-2.5,-2.5) |-(Controller.180);
        \draw[->,blue] (plant.-22) node[xshift = 1.25em, yshift = -0.75em]{$\xi_k$} -| +(0.65,-1.25) |-(DMA.-10);
        
        \draw[blue,->] (Controller.0) -| +(0.4,-1) node[xshift = -0.75em, yshift = 0.1em]{$u_k$} |- (DMA.10);
        \draw[blue,->] (DMA.180) node[xshift = 0.25em, yshift = 2em]{$A_k, B_k$} -| +(-0.5,1)  |- (Controller.200);

    \end{tikzpicture}
    }
        \caption{Dynamic Mode Adaptive Control (DMAC) architecture for model-free, data-driven, and learning-based control of dynamic systems.         
        }
        \label{fig:DMAC_architecture}
    \end{figure}

\subsection{Dynamic Mode Approximation}
\label{sec:DynApprox}
Let $\xi_k \in \BBR^{l_\xi}$ denote the measured portion of the state of the system.
Note that $\xi_k$ may or may not be the entire state of the system.
To compute the control signal $u_k$, we first approximate linear maps $A \in \BBR^{l_\xi \times l_\xi}$ and $B \in \BBR^{l_\xi \times l_u}$ such that 
\begin{align}
    \xi_{k+1} = A \xi_{k} + B u_{k},
    \label{eq:linear_approximation}
\end{align}
which can be reformulated as
\begin{align}
    \xi_{k+1} = \Theta \phi_k.
    \label{eq:linear_approximation_AxForm_1}
\end{align}
where 
\begin{align}
    \Theta &\isdef \matl A  & B \matr \in \BBR^{l_\xi \times (l_\xi + l_u)}, 
    \quad
    \phi_k \isdef \matl \xi_{k} \\ u_{k} \matr \in \BBR^{l_\xi+l_u}. 
\end{align}
A matrix $\Theta$ such that \eqref{eq:linear_approximation_AxForm_1} is satisfied may not exist. 
However, an approximation of such a matrix can be obtained by minimizing 
\begin{align}
    J_k (\Theta)
        \isdef 
            \sum_{i=0}^{k} &\lambda^{k-i} \| \xi_{k} - \Theta \phi_{k-1} \|^2_2 
            +
            \lambda^k \tr (\Theta^\rmT R_\Theta \Theta) ,
    \label{eq:J_k_def}
\end{align}
where
$R_\Theta \in \BBR^{(l_\xi+l_u) \times (l_\xi+l_u)}$ is a positive definite regularization matrix that ensures the existence of the minimizer of \eqref{eq:J_k_def}
and 
$\lambda \in (0,1]$ is a forgetting factor. 
In nonlinear or time-varying systems, $\Theta$ approximated by minimizing \eqref{eq:J_k_def} in the case where the state varies significantly in the state space may not be able to capture the local linear behavior at the current state.
Thus, incorporating a geometric forgetting factor to prioritize recent data over older data improves the linear approximation and prevents the algorithm from becoming sluggish.

\begin{proposition}
    Consider the cost function \eqref{eq:J_k_def}.
    For all $k\geq 0,$ define the minimizer of \eqref{eq:J_k_def} as
    \begin{align}
        \Theta_k 
            \isdef 
                \min_{\Theta \in \BBR^{l_x \times (l_x + l_u)}} J_k(\Theta).
    \end{align}
    Then, the minimizer $\Theta_k$ satisfies
    \begin{align}
        \Theta_k
            &=
                \Theta_{k-1} 
                +
                \left(
                    \xi_{k} - \Theta_{k-1} \phi_{k-1}
                \right)
                \phi_{k-1}^\rmT \SP_k  
            , \\
        \SP_k
            &=
                \lambda \inv \SP_{k-1} 
                -
                \lambda \inv
                \SP_{k-1} \phi_{k-1}
                \gamma_k \inv 
                \phi_{k-1}^\rmT \SP_{k-1},
    \end{align}
    where, for all $k \geq 0,$ 
    $\gamma_k \isdef \lambda  +  \phi_{k-1}^\rmT \SP_{k-1} \phi_{k-1},$ and  
    $\Theta_0 = 0,$
    $\SP_0 \isdef R_\Theta\inv. $
\end{proposition}

\begin{proof}
    See Proposition V.2 in \cite{oveissi2025modelfreedynamicmodeadaptive}.
\end{proof}

Note that the cost function \eqref{eq:J_k_def} is a matrix extension of the cost function typically considered in engineering applications \cite{goel2020recursive}.
As shown in \cite{Mareels1986,Mareels1988,goel2020recursive}, persistency of excitation is required to ensure that 1) the estimate converges and 2) the corresponding covariance matrix $\SP_k$ remains bounded. 
To ensure the persistency of excitation, in this paper, we introduce a zero-mean white noise in the control signal to promote persistency in the regressor $\phi_k$, as discussed in Section \ref{sec:controlUpdate}.

\subsection{Tracking Controller}
\label{sec:controlUpdate}
This subsection presents the algorithm to compute the control signal $u_k$ using the dynamics approximation computed in Section \ref{sec:DynApprox}.
To track the reference signal $r_k  \in \BBR$, the DMAC algorithm uses the fullstate feedback controller with integral action.
Note that the full state refers to the state $\xi_k$ and not the system state $x_k.$ 
In this work, we assume that 
\begin{align}
    \xi_k 
        \isdef
            \matl P_{t4,k} \\ X_{\text{CO},k} \\ y_k \matr \in \BBR^3,
\end{align}
where 
$P_{t4,k}$ is the normalized total combustor pressure, $X_{\text{CO},k}$ is the normalized carbon monoxide concentration in the exhaust, and $y_k$ is the normalized thrust.
The concentration of carbon monoxide is determined from the equilibrium composition, calculated using the equilibrium Gibbs solver and 20-species pressure-comprehensive skeletal kinetics mechanism.  
The total combustor pressure and thrust are calculated from Equations \eqref{pressuredrop} and \eqref{thrustcalc}, respectively. 
The normalization is performed using a symmetric min-max transformation that maps each physical variable to the range $[-1, 1]$, ensuring similar magnitude across all state variables.
This transformation improves the numerical conditioning of the recursive least squares update used in the DMAC dynamic mode approximation and ensures that the regressor $\phi_k$ remains well-conditioned throughout the learning process.


The control law is 
\begin{align}
    u_k = K_{\xi,k} \xi_k + K_{q,k} q_k + v_k,
\end{align}
where
the matrices $K_{\xi,k} \in \BBR^{l_u \times l_\xi } $ and $K_{q,k} \in \BBR^{l_u \times l_y}$ are the time-varying fullstate feedback gain and the integrator gain, computed using the technique shown in Appendix A of \cite{oveissi2025modelfreedynamicmodeadaptive} and $v_k \sim \SN(0,\sigma_v I_{l_u})$ is a zero-mean white noise signal added to the control to promote persistency in the regressor $\phi_k$ used in the dynamic mode approximation step.
Note that the integrator state $q_k$ satisfies
\begin{align}
    q_{k+1}
        =
            q_k + z_k,
\end{align}
where $z_k \isdef r_k - y_k$ is the output error. 

The gain matrices $K_{\xi,k}$ and $K_{q,k}$ are computed using the well-known linear-quadratic-integral control \cite{young1972approach}, which requires the $A, B, $ and $C$ matrices of the system. 
In this work, we use the MATLAB's \texttt{lqi} routine to compute the gain matrices $K_{\xi,k}$ and $K_{q,k}.$
The dynamics matrix $A$ and the input matrix $B$ are given by the dynamic mode approximation described in \ref{sec:DynApprox}.
Note that the output matrix is given by
\begin{align}
    C = \matl 0 & 0 & 1 \matr,
\end{align}
which reflects the fact that the thrust is the third entry in the DMAC state vector and is directly selected as the system output.

\section{Simulation Results}
\label{sec:prelim_results}
This section presents numerical examples that demonstrate the application of the DMAC technique for regulating the thrust generated by the SFRJ. The DMAC hyperparameters were selected through a simple grid search to achieve a satisfactory transient response under nominal conditions.

The control signal $u_k$ is used to modulate the capture radius $r_0$ as
\begin{align}
    r_0 = \overline{r}_0 - 0.001 u_k,
\end{align}
where $\overline{r}_0 = 53.58$~mm is the nominal capture radius. The scaling factor ensures that the regressor components used for dynamic mode estimation have similar magnitudes, thereby improving numerical stability during recursive least squares updates.

\subsection{Command Following}
First, we consider the problem of regulating the SFRJ thrust to a constant value. 
In particular, the SFRJ is commanded to generate a constant thrust value of $r = 100 $ $\rmN$. 
%
Since $\xi_k \in \BBR^3$ and $u_k \in \BBR,$ it follows that $\Theta_k$ is a $3 \times 4$ matrix. 
In DMAC, we set $R_\Theta = 10^2 I_4$ and the forgetting factor $\lambda = 0.999$.
The LQR weights in the tracking controller are $R_1 = I_4$  and $R_2 = 1.$

Figure \ref{fig:DMAC_SFRJ_NCSU_steps_1000_lambda_0_999_R0_1e2_Q_1_R_1_sysdim_3_ref_SingleStep_100_noise_1e_neg3_NNthrust} shows the closed-loop response of the SFRJ with the DMAC algorithm updating the controller, where
a) shows the commanded thrust $r=100$ $\rmN$ and the generated thrust $y_k,$ b) shows the control signal $u_k,$ c) shows the absolute value of the tracking error $z_k \isdef y_k - r$ on a logarithmic scale, and d) shows the estimate matrix $\Theta_k$ computed by DMAC.
Note that the thrust output tracks the commanded thrust. 

\begin{figure}[H]
    \centering
    \includegraphics[width=0.7\columnwidth]{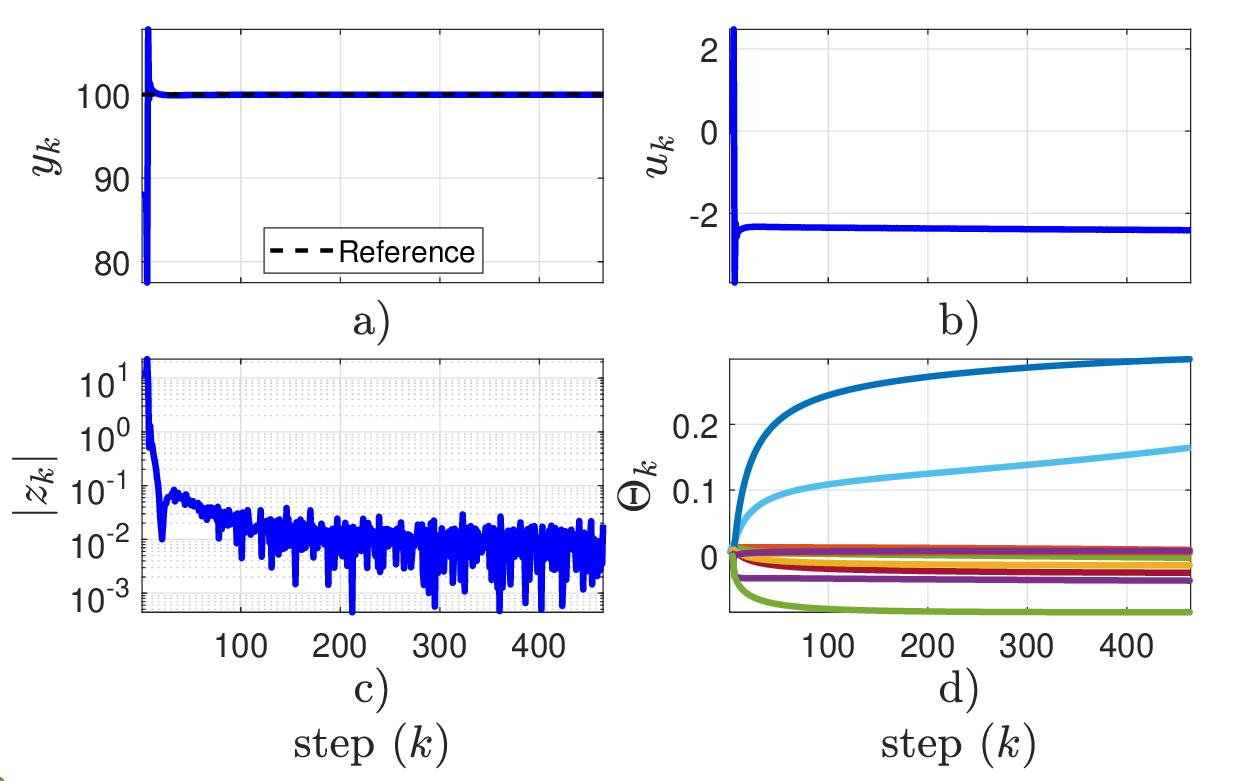}
    \caption{Closed-loop step response of SFRJ with DMAC.
    } 
    \label{fig:DMAC_SFRJ_NCSU_steps_1000_lambda_0_999_R0_1e2_Q_1_R_1_sysdim_3_ref_SingleStep_100_noise_1e_neg3_NNthrust}
\end{figure}

Next, the SFRJ is commanded to follow a double-step command. 
Specifically, the thrust command is $r = 100 $ $\rmN$ for $k \in (0,200)$ and $r = 110$ $\rmN$ for $k \geq 200.$
We emphasize that the DMAC hyperparameters are kept the same as in the previous case.  
Figure \ref{fig:DMAC_SFRJ_NCSU_steps_1000_lambda_0_999_R0_1e2_Q_1_R_1_sysdim_3_ref_DoubleStep_100_110_noise_1e_neg3_NNthrust} shows the closed-loop response of the SFRJ with the DMAC algorithm updating the controller, where
a) shows the commanded thrust $r$ and the generated thrust $y_k,$ b) shows the control signal $u_k,$ c) shows the absolute value of the tracking error $z_k \isdef y_k - r$ on a logarithmic scale, and d) shows the estimate matrix $\Theta_k$ computed by DMAC. Note that the output error approaches zero.

\begin{figure}[H]
    \centering
    \includegraphics[width=0.7\columnwidth]{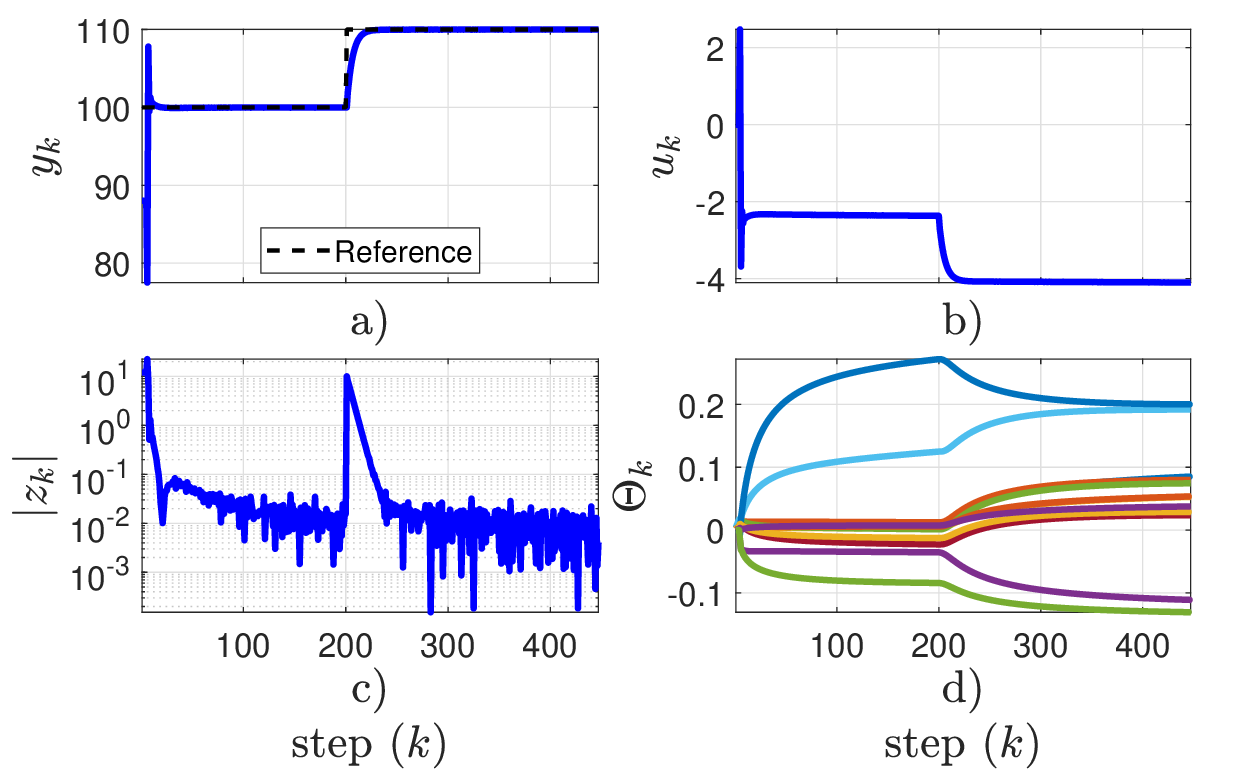}
    \caption{Closed-loop response of SFRJ with DMAC to double step reference.} 
    \label{fig:DMAC_SFRJ_NCSU_steps_1000_lambda_0_999_R0_1e2_Q_1_R_1_sysdim_3_ref_DoubleStep_100_110_noise_1e_neg3_NNthrust}
\end{figure}

\subsection{Controller Sensitivity Analysis}
A sensitivity analysis was performed to show the robustness of the control algorithm against variations in the DMAC hyperparameters and changes in the SFRJ model, as reflected in the closed-loop performance.

\subsubsection{Effect of Hyperparameters}
First, to investigate the robustness of the DMAC algorithm to its tuning hyperparameters $R_\Theta, \lambda, R_1$ and $R_2,$  we vary each of the hyperparameters systematically by keeping other hyperparameters at their nominal values.
Figure~\ref{fig:DMAC_SFRJ_NCSU_steps_1000_DMAC_Sensitivity_sysdim_3_ref_SingleStep_100_noise_1e_neg3_NNthrust}
shows the effect of the DMAC hyperparameters on the closed-loop response $y_k$ and the absolute tracking error $|z_k|$.
From top to bottom, the rows correspond to variations in $R_\Theta$, $\lambda$, $R_1$, and $R_2$, respectively. Note that, in each case, the hyperparameter is varied by a few orders of magnitude, suggesting that DMAC is robust to tuning parameters.

\begin{figure}[H]
    \centering
    \includegraphics[width=0.7\columnwidth]{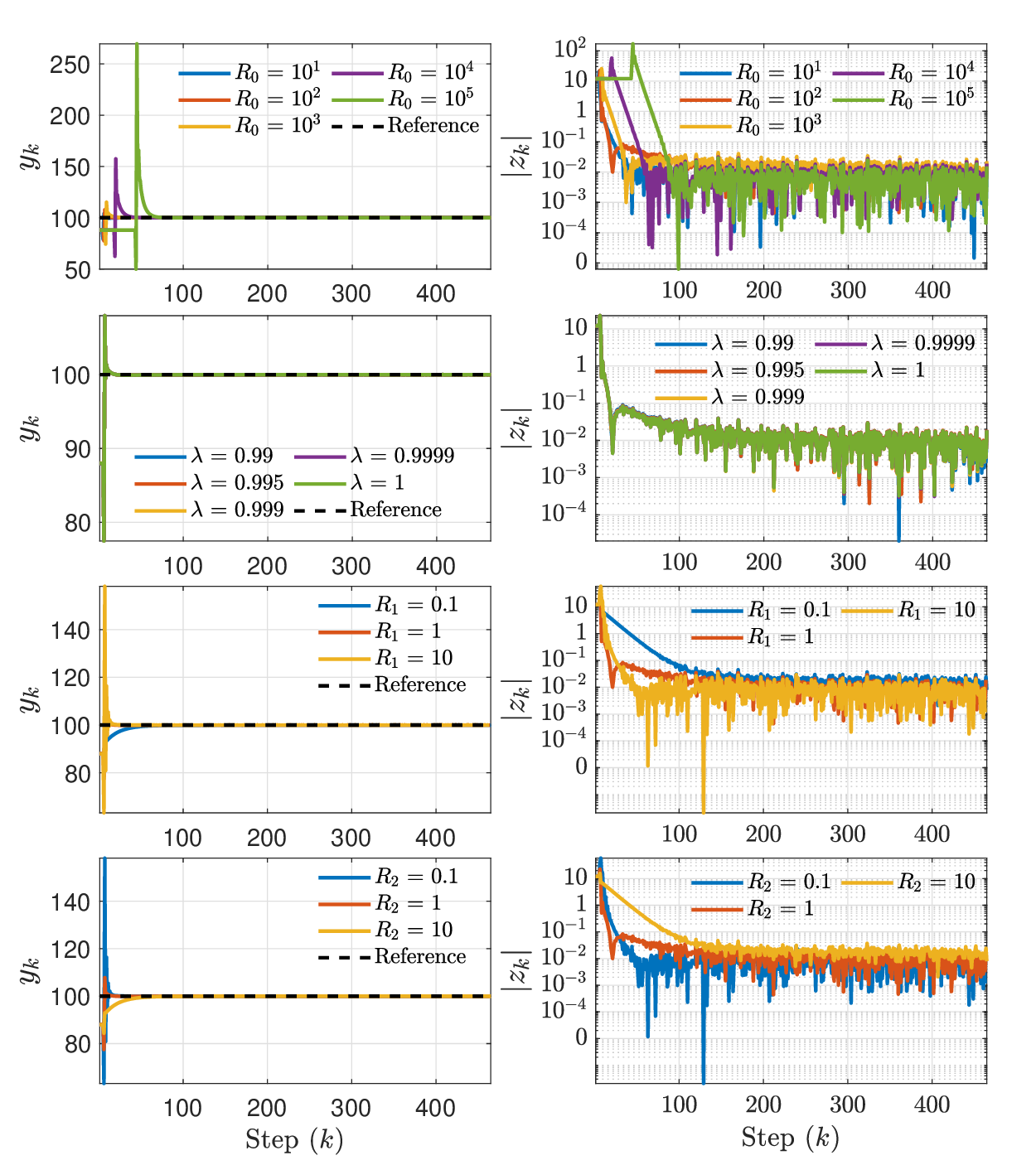}
    \caption{Effect of DMAC hyperparameters on the closed-loop performance.} 
    \label{fig:DMAC_SFRJ_NCSU_steps_1000_DMAC_Sensitivity_sysdim_3_ref_SingleStep_100_noise_1e_neg3_NNthrust}
\end{figure}

\subsubsection{Monte Carlo Testing of Controller Robustness}
Next, the robustness of the DMAC controller to variations in SFRJ model performance was evaluated using the Monte Carlo method (MCM). 
Two separate tests were conducted to assess reliable tracking performance across different engine architectures and altitude–thrust command conditions.

In the first test, the SFRJ model parameters $\alpha$ and $\eta_c$ were modeled as Gaussian random variables and sampled prior to executing a step command-following task, in which the SFRJ was commanded to produce a constant thrust of $100$~N at an altitude of $H = 30$~km.
These parameters were chosen for randomization because thrust generation is highly sensitive to the propellant regression rate and combustion efficiency, which are governed by $\alpha$ and $\eta_c$, respectively.
The mean values of $\alpha$ and $\eta_c$ were $4.44\times10^{-7}$ and $0.75$, respectively, with corresponding standard deviations of $4.44\times10^{-8}$ and $0.05$.
A total of 200 Monte Carlo trials were performed, each using a distinct quasi-static SFRJ model realization, to evaluate controller robustness under these uncertainties.

The Monte Carlo results are shown in Figure \ref{fig:MCM}. Of the 200 randomly sampled quasi-static models, 195 simulations converged, corresponding to a 97.5\% success rate, demonstrating the robustness of the DMAC algorithm to variations in SFRJ model parameters. 
Although the variations in $\alpha$ and $\eta_c$ produced significant differences in the initial thrust estimates, the controller consistently achieved accurate step command tracking. 
Note that no retuning of the DMAC hyperparameters was performed across the sampled models; that is, all 200 simulations were executed using the same controller configuration.

\begin{figure}[h]
    \centering
    \includegraphics[width=0.49\columnwidth]{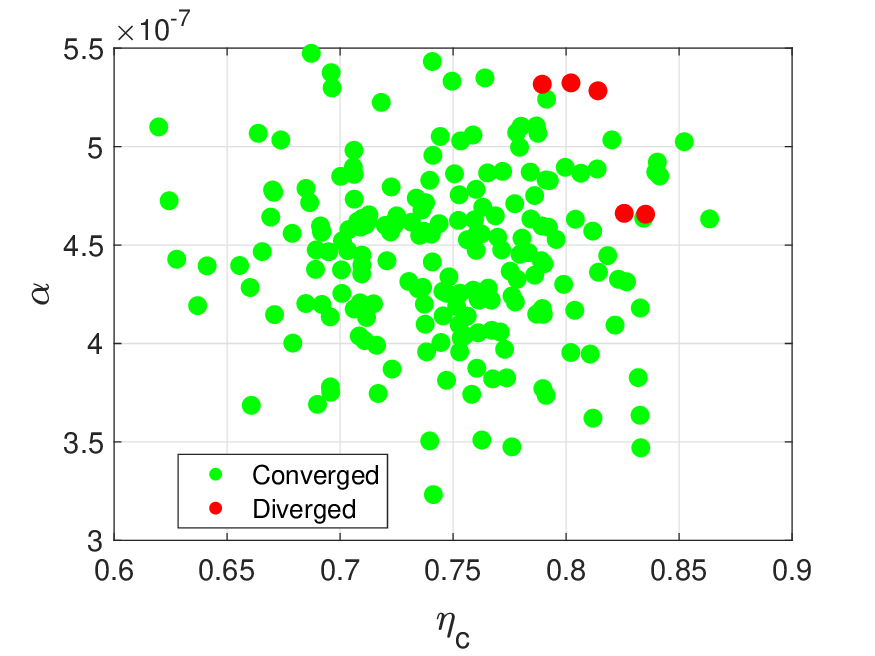}
    \includegraphics[width=0.49\columnwidth]{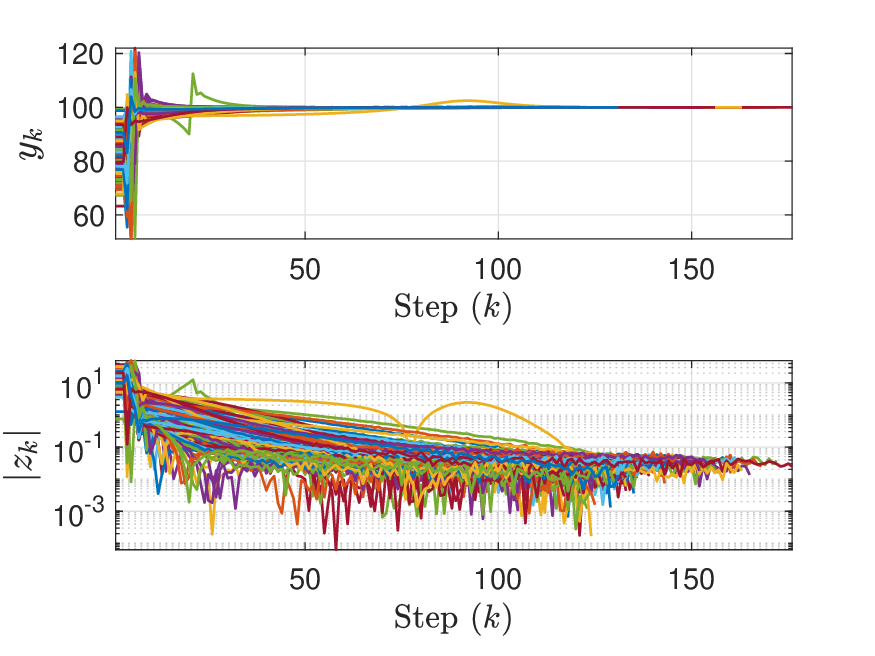}
    \caption{Demonstration of DMAC robustness to SFRJ model performance: a) scatterplot of solution convergence and b) time history of thrust tracking (only converged solutions)}
    \label{fig:MCM}
\end{figure}

In the second test, the SFRJ model parameters were fixed at their nominal values, 
$\alpha = 4.44 \times10^{-7}$ and $\eta_c = 0.75,$ while the engine was commanded to track randomly generated thrust profiles at randomly selected altitudes. 
The altitude $H$ was sampled uniformly from 23 km to 36 km, and for each selected altitude, five thrust commands were randomly chosen within the model’s feasible thrust range, excluding a 5\% margin from the minimum and maximum limits. 
A total of 200 Monte Carlo trials were performed to evaluate the controller’s ability to track diverse thrust–altitude command pairs in the step command-following task.

\begin{figure}[h]
    \centering
    \includegraphics[width=0.49\columnwidth]{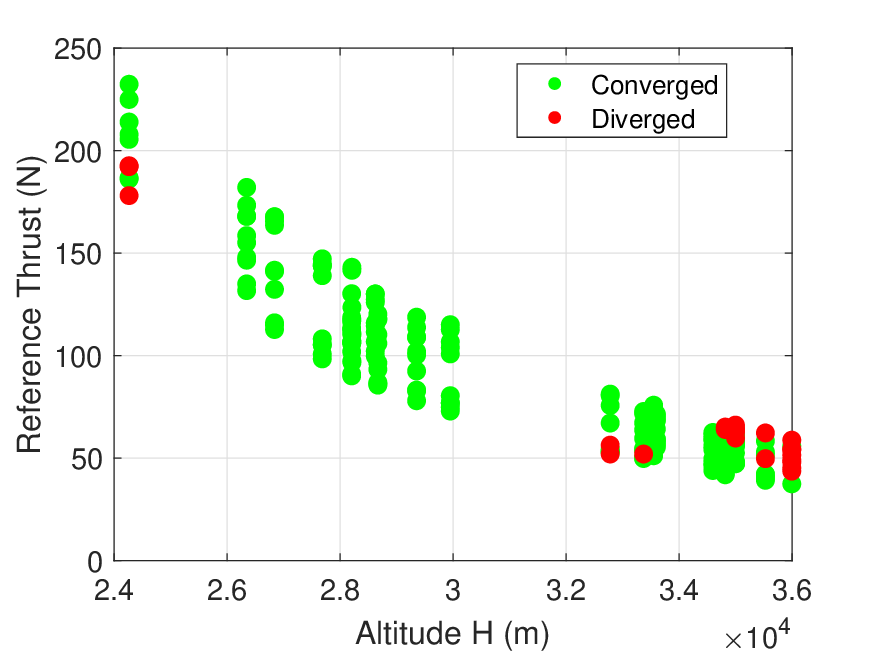}
    \vspace{-1em}
    \includegraphics[width=0.49\columnwidth]{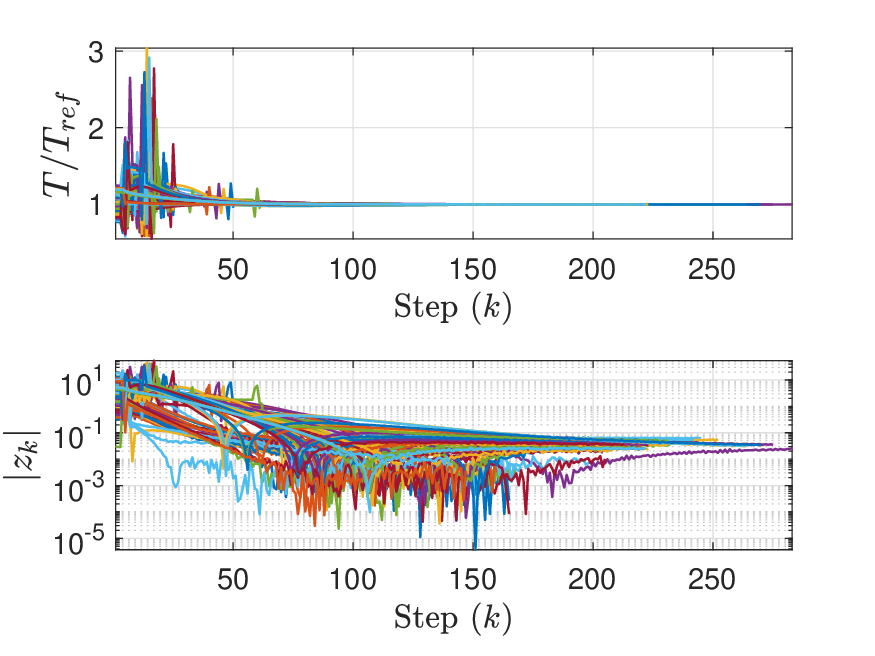}
    \caption{Demonstration of DMAC robustness to altitude-thrust commanding: a) scatterplot of solution convergence and b) time history of normalized thrust tracking (only converged solutions)}
    \label{fig:MCMH}
\end{figure}

The results of the second Monte Carlo trial are shown in Figure \ref{fig:MCMH}. 
Of the 200 randomly generated thrust–altitude pairs, 177 simulations converged, corresponding to a 88.5\% success rate. 
These results further demonstrate the robustness of the DMAC algorithm across a broad hypothetical operational envelope. Divergence was observed primarily for low thrust commands at lower altitudes, whereas performance remained consistently accurate in the higher thrust regime $(50 \ \rmN <T< 200 \ \rmN)$.
This indicates that retuning of DMAC hyperparameters is necessary only for operating conditions well beyond the nominal range, and only a limited number of tuned configurations would be required to ensure reliable tracking across the entire operational window.
As with the previous trial, the same DMAC configuration was used for all 200 simulations, highlighting the algorithm’s inherent robustness compared to traditional fixed-gain controllers, which typically require manual retuning under model uncertainty or parameter variations.

\section{Conclusion}
\label{sec:conclusions}

This paper investigated the application of the dynamic mode adaptive control (DMAC) framework to regulate the thrust generated by a quasi-static one-dimensional model of an SFRJ with variable cowl geometry. 
The DMAC framework used limited measurements of combustor pressure, carbon monoxide concentration in the exhaust, and the thrust to synthesize a tracking controller in order to regulate the SFRJ thrust.
The simulation results indicate that the DMAC framework can successfully regulate the SFRJ thrust without requiring a detailed combustion dynamics model, instead using only limited online measurements. 
The additional numerical studies presented in this manuscript further evaluated the performance and robustness of the DMAC-based controller.
A sensitivity analysis showed that the closed-loop response is relatively insensitive to variations in the DMAC hyperparameters, indicating that the controller maintains satisfactory performance over a wide range of tuning selections. 
A comprehensive Monte Carlo evaluation assessed robustness to uncertainties in propellant regression and combustion efficiency, as well as variations in altitude and commanded thrust.
Across these tests, the controller achieved high convergence rates without hyperparameter retuning, highlighting the inherent robustness of the DMAC framework to model variability and changes in operating conditions.

\section{Acknowledgment}
This research was supported by the Office of Naval Research grant N00014-23-1-2468.

\bibliography{Paperbib}
\end{document}

%% file: PackagesCommands.tex
\usepackage{amsmath} 
\usepackage{amsfonts}
\usepackage{amsthm}

\newtheorem{proposition}{Proposition}[section]

\usepackage{float} 

\usepackage{subcaption}

\usepackage{tikz}
\usetikzlibrary{circuits.ee.IEC}
\usetikzlibrary{shapes,arrows,calc,positioning}
\tikzstyle{bigblock} = [draw, fill=blue!20, rectangle, 
    minimum height=6em, minimum width=8em]
\tikzstyle{medblock} = [draw, fill=blue!20, rectangle, 
    minimum height=4em, minimum width=4em]    
\tikzstyle{mux} = [draw, fill=black!20, rectangle, 
    minimum height=5em, minimum width=0.1em]    
\tikzstyle{smallblock} = [draw, fill=blue!20, rectangle, 
    minimum height=2em, minimum width=3em]
    
\tikzstyle{data_block} = [draw, fill=green!20, rectangle, 
    minimum height=2em, minimum width=3em]
\tikzstyle{ops_block} = [draw, fill=blue!20, rectangle, 
    minimum height=2em, minimum width=3em]    
\tikzstyle{est_block} = [draw, fill=red!20, rectangle, 
    minimum height=2em, minimum width=3em]    
    
\tikzstyle{sum} = [draw, fill=blue!20, circle, node distance=1cm,minimum height=0.5cm]
\tikzstyle{signal} = [coordinate]
\tikzstyle{pinstyle} = [pin edge={to-,thin,black}]
\tikzstyle{block} = [draw, fill=blue!20, rectangle, 
    minimum height=3em, minimum width=9em]
\tikzstyle{blockS} = [draw, fill=blue!20, rectangle, 
    minimum height=3em, minimum width=4em]    
\tikzstyle{input} = [coordinate]
\tikzstyle{output} = [coordinate]

\tikzstyle{gain} = [draw, fill=blue!20, regular polygon, regular polygon sides=3, shape border rotate=30]
\tikzstyle{gain_vert} = [draw, fill=blue!20, regular polygon, regular polygon sides=3, shape border rotate=0]

\usetikzlibrary{matrix}

\usetikzlibrary{positioning}
\usetikzlibrary{math}

\usepackage{hyperref}
\usepackage{xcolor}
\hypersetup{
    colorlinks,
    linkcolor={blue!100!black},
    citecolor={blue!50!black},
    urlcolor={blue!80!black}
}

\newcommand{\bc}{\begin{center}}
\newcommand{\ec}{\end{center}}
\newcommand{\benum}{\begin{enumerate}}
\newcommand{\eenum}{\end{enumerate}}

\newcommand{\matl}{\left[ \begin{array}}
\newcommand{\matr}{\end{array} \right]}

\renewcommand{\matl}{\begin{bmatrix}}
\renewcommand{\matr}{\end{bmatrix}}

\newcommand{\matls}{\left[ \begin{smallmatrix}}
\newcommand{\matrs}{\end{smallmatrix} \right]}
\newcommand{\isdef}{\stackrel{\triangle}{=}}

\newcommand{\inv}{^{-1}}

\newcommand{\tr}{{\rm tr}\,}

\newcommand{\rmN}{{\rm N}}

\newcommand{\rmT}{{\rm T}}

\newcommand{\BBR}{{\mathbb R}}

\newcommand{\SN}{{\mathcal N}}

\newcommand{\SP}{{\mathcal P}}

\usepackage{enumitem}
\newlist{todolist}{itemize}{2}
\setlist[todolist]{label=$\square$}
\usepackage{pifont}